\documentclass[12pt]{amsart} 
\usepackage{amsmath,amssymb,amscd,amsthm}
\usepackage{latexsym}
\usepackage[english]{babel}
\usepackage[latin1]{inputenc}       
\usepackage{pgf,pgfarrows,pgfnodes,pgfautomata,pgfheaps}

\usepackage[a4paper,twoside,left=3cm,right=2.8cm,top=3.1cm,bottom=2.3cm]{geometry}

\newtheorem{thm}{Theorem}
\newtheorem{cor}[thm]{Corollary}
\newtheorem{Claim}[thm]{Claim}
\newtheorem{theorem}{Theorem}[section]
\newtheorem{proposition}[theorem]{Proposition}
\newtheorem{lemma}[theorem]{Lemma}

\newtheorem{definition}[theorem]{Definition}

\newtheorem{remark}[theorem]{Remark}

\def\irr#1{{\rm Irr}(#1)}
\def\irrr#1#2 {\irr {#1 \mid #2}}
\newcommand{\R}{\mathbb R}
\newcommand{\N}{\mathbb N}

\newcommand{\sfe}{{{\mathbb S}^{n-1}}}

\begin{document}

\title[Gaussian Brunn-Minkowski]{On the Gardner-Zvavitch conjecture: symmetry in inequalities of Brunn-Minkowski type}
\author[Alexander V. Kolesnikov, Galyna V. Livshyts]{Alexander V. Kolesnikov, Galyna V. Livshyts}
\address{National Research University Higher School of Economics, Russian Federation}
\email{sascha77@mail.ru}
\address{School of Mathematics, Georgia Institute of Technology, Atlanta, GA} \email{glivshyts6@math.gatech.edu}

\subjclass[2010]{Primary: 52} 
\date{\today}
\begin{abstract}
In this paper, we study the conjecture of Gardner and Zvavitch from \cite{GZ}, which suggests that the standard Gaussian measure $\gamma$ enjoys $\frac{1}{n}$-concavity with respect to the Minkowski addition of \emph{symmetric} convex sets. We prove this fact up to a factor of 2: that is, we show that for symmetric convex $K$ and $L,$ and $\lambda\in [0,1],$
$$
\gamma(\lambda K+(1-\lambda)L)^{\frac{1}{2n}}\geq \lambda \gamma(K)^{\frac{1}{2n}}+(1-\lambda)\gamma(L)^{\frac{1}{2n}}.
$$
More generally, this inequality holds for convex sets containing the origin.
Further, we show that under suitable dimension-free uniform bounds on the Hessian of the potential, the log-concavity of even measures can be strengthened to $p$-concavity, with $p>0,$ with respect to the addition of symmetric convex sets.
\end{abstract}
\maketitle

\section{Introduction}

Throughout this paper, we work in $n$-dimensional Euclidean space, which we denote by $\R^n$. The unit ball in $\R^n$ will be denoted by $B_2^n$ and the unit sphere by $\sfe$. The Lebesgue measure of a measurable set $A\subset \R^n$ is denoted by $|A|$. 

Recall that a Borel measure $\mu$ on $\R^n$ is called \emph{log-concave} if for every pair of Borel sets $K$ and $L$, 
\begin{equation}\label{logconc}
\mu(\lambda K+(1-\lambda) L)\geq \mu(K)^{\lambda}\mu(L)^{1-\lambda}.
\end{equation}

More generally, $\mu$ is called \emph{$p$-concave} for $p\geq 0,$ if
\begin{equation}\label{pconc}
\mu(\lambda K+(1-\lambda)L)^p\geq \lambda\mu(K)^{p}+(1-\lambda)\mu(L)^{p}.
\end{equation}
Log-concavity corresponds to the limiting case $p=0.$ By H\"older's inequality, if $p>q\geq 0,$ and a measure is $p$-concave, it is also $q$-concave.

Borell's theorem ensures that a measure with a log-concave density is log-concave \cite{bor}. Further, the celebrated Brunn-Minkowski inequality states that for all Borel sets $K$ and $L$, and for every $\lambda\in [0,1]$,
\begin{equation}\label{BM-add}
|\lambda K+(1-\lambda)L|^{\frac{1}{n}}\geq \lambda |K|^{\frac{1}{n}}+(1-\lambda)|L|^{\frac{1}{n}}.
\end{equation}
See more on the subject in Gardner's survey \cite{Gar}, and some classical textbooks in Convex Geometry, e.g.  Bonnesen, Fenchel \cite{BF}, Schneider \cite{book4}. In view of H\"older's inequality, (\ref{BM-add}) implies the log-concavity of the Lebesgue measure:
\begin{equation}\label{BM-mult}
|\lambda K+(1-\lambda)L|\geq |K|^{\lambda}|L|^{1-\lambda}.
\end{equation}
The homogeneity of the Lebesgue measure ensures that, in fact, (\ref{BM-mult}) is equivalent to (\ref{BM-add}). However, this is not the case for general (non-homogeneous) measures $\mu$ on $\R^n$: the log-concavity property (\ref{logconc}) does not imply the stronger inequality
\begin{equation}\label{dimbm}
\mu(\lambda K+(1-\lambda)L)^{\frac{1}{n}}\geq \lambda \mu(K)^{\frac{1}{n}}+(1-\lambda)\mu(L)^{\frac{1}{n}}.
\end{equation}
In fact, (\ref{dimbm}) cannot hold in general for a probability measure: if $K$ is fixed, and $L$ is shifted far away from the origin, then the left hand side of (\ref{dimbm}) is close to zero (thanks to the decay of the measure at infinity), while the right hand side
is bounded from below by a positive constant. 

Gardner and Zvavitch conjectured \cite{GZ} that for the standard Gaussian measure $\gamma$, any pair of \emph{symmetric} convex sets $K$ and $L,$ and any $\lambda\in[0,1]$, one has
\begin{equation}\label{GZ0}
\gamma(\lambda K+(1-\lambda)L)^{\frac{1}{n}}\geq \lambda \gamma(K)^{\frac{1}{n}}+(1-\lambda)\gamma(L)^{\frac{1}{n}}.
\end{equation}
In fact, initially they considered the possibility that (\ref{GZ0}) may hold for sets $K$ and $L$ containing the origin, but a counterexample to that was constructed by Nayar and Tkocz \cite{NaTk}. 

Symmetry seems to play a crucial role in the improvement of isoperimetric type inequalities. One simple example when such a phenomenon occurs is the Poincar\'e inequality (sometimes also referred to as Wirtinger's inequality): for any $C^1$-smooth $2\pi$-periodic function $\psi$ on $\R,$
\begin{equation}
\label{1dimpoincare}\frac{1}{2\pi}\int_{-\pi}^{\pi} \psi^2 \ \,dx-\left(\frac{1}{2\pi}\int_{-\pi}^{\pi} \psi \ \,dx\right)^2\leq \frac{1}{2\pi}\int_{-\pi}^{\pi} (\psi')^2 \ \,dx,
\end{equation}
and in the case when $\psi$ is also $\pi$-periodic  one has the stronger inequality
\begin{equation}
\label{1dimpoincare-even}\frac{1}{2\pi}\int_{-\pi}^{\pi} \psi^2 \ \,dx-\left(\frac{1}{2\pi}\int_{-\pi}^{\pi} \psi \ \,dx\right)^2\leq \frac{1}{8\pi}\int_{-\pi}^{\pi} (\psi')^2 \ \,dx.
\end{equation}
Note that $\pi$-periodicity of $\psi$ is equivalent to the property that $\psi$ is an even function on $S^1$ after identification of the circle $S^1$ with $[0, 2\pi)$ under the mapping $t \to e^{it}$.  
The standard proof of (\ref{1dimpoincare}) applies the Fourier series expansion:
$$
\psi(x) = \frac{a_0}{\sqrt{2\pi}} + \sum_{n=1}^{\infty} \frac{a_n \cos nx}{\sqrt{\pi}} + \frac{b_n \sin nx}{\sqrt{\pi}}.
$$
Observe that $\int_{-\pi}^{\pi} \psi(x) dx = \sqrt{2 \pi} a_0$, $\int_{-\pi}^{\pi} \psi^2 dx= a^2_0 + \sum_{n=1}^{\infty} a^2_n + b^2_n$, and  $\int_{-\pi}^{\pi} (\psi ')^2 dx= \sum_{n=1}^{\infty} n(a^2_n + b^2_n)$. This implies (\ref{1dimpoincare}). 
If, in addition, $\psi$ is $\pi$-periodic, then 
$$\int_0^{2\pi} \psi(x) \cos x dx = \int_0^{2\pi} \psi(x) \sin x dx =0,$$ 
and we get  (\ref{1dimpoincare-even}). See, e.g., Groemer \cite{Groemer}, Theorem 4.4.1 on page 149.

In general, given a log-concave probability measure $\mu$ with density $e^{-V}$ such that 
$$\nabla^2V\geq k_1 {\rm Id},\  \ \ \ k_1 >0,$$ 
 one has, for any $C^1$ function $\psi,$
\begin{equation}\label{poinc}
\int \psi^2 \ d \mu-\left(\int \psi \ d \mu \right)^2\leq \frac{1}{k_1}\int |\nabla \psi|^2 \ d \mu;
\end{equation}
this follows from the Brascamp-Lieb inequality \cite{BrLi}. Cordero-Erasquin, Fradelizi and Maurey \cite{CFM} proved a strengthening of (\ref{poinc}). This strengthening implies, in particular, that if $\psi$ and $V$ are additionally even, then
\begin{equation}\label{poinc-even}
\int \psi^2 \ d \mu-\left(\int \psi \ d \mu \right)^2\leq \frac{1}{2k_1}\int |\nabla \psi|^2 \ d \mu.
\end{equation}

In recent years, a number of conjectures have  appeared concerning the improvement of inequalities of Brunn-Minkowski type under additional symmetry assumptions. For instance, in the case of the Gaussian measure, Schechtman, Schlumprecht and Zinn \cite{SSZ} obtained an exciting inequality in the style of the conjecture of Dar \cite{dar}; Tehranchi \cite{tehr} has recently found an extension of their results, which is also a strengthening of the famous Gaussian correlation conjecture, recently proved by Royen \cite{royen} (see also Lata\l{}a, Matlak \cite{LatMat}).

One of the most famous of these conjectures is the Log-Brunn-Minkowski conjecture of B\"or\"oczky, Lutwak, Yang and Zhang (see \cite{BLYZ}, \cite{BLYZ-1}, \cite{BLYZ-2}). It states that for all symmetric convex bodies $K$ and $L$ with support functions $h_K$ and $h_L,$
\begin{equation}\label{Log-BM_eq}
|\lambda K+_0 (1-\lambda)L|\geq |K|^{\lambda}|L|^{1-\lambda},
\end{equation}
where $+_0$ stands for the \emph{geometric mean} 
\begin{equation}
\lambda K+_0 (1-\lambda)L=\{x\in \R^n\,:\, \langle x,u\rangle\leq h_K(u)^{\lambda} h_L(u)^{1-\lambda}\,\,\forall u\in \sfe\}.
\end{equation}
B\"or\"oczky, Lutwak, Yang and Zhang \cite{BLYZ} showed that the Log-Brunn-Minkowski conjecture holds for $n=2$. Saroglou \cite{christos} and Cordero-Erasquin, Fradelizi, Maurey \cite{CFM}  proved that (\ref{Log-BM_eq}) is true when $K$ and $L$ are unconditional (that is, they are symmetric with respect to every coordinate hyperplane). The conjecture was verified in a neighborhood of the Euclidean ball by Colesanti, Livshyts and Marsiglietti \cite{CLM}, \cite{CL}. In \cite{KolMilsupernew}, Kolesnikov and E. Milman found a relation between the Log-Brunn-Minkowski conjecture and the second eigenvalue problem for certain elliptic operators. In addition,  the ``local version'' of the Log-Brunn-Minkowski conjecture was verified in  \cite{KolMilsupernew}  for the cube and for $l_q$-balls, $q \ge 2$, when the dimension is sufficiently large. By ``local version'', we mean an inequality of isoperimetric or Poincar\'e type, obtained by differentiating the   inequality on an appropriate family of convex sets. Building on the results from \cite{KolMilsupernew}, Chen, Huang, Li, Liu \cite{global} managed to verify the $L_p$-Brunn-Minkowski inequality for symmetric sets, using techniques from PDE. Saroglou \cite{christos1} showed that the validity of (\ref{Log-BM_eq}) for all convex bodies is equivalent to the validity of the analogous statement for an arbitrary log-concave measure.

In \cite{LMNZ},  Livshyts, Marsiglietti, Nayar and Zvavitch proved that the Log-Brunn-Minkowski conjecture implies the conjecture of Gardner and Zvavitch. In fact, if (\ref{Log-BM_eq}) was proved to be true, then (\ref{dimbm}) would hold for any even log-concave measure $\mu$ and for all symmetric convex $K$ and $L.$ Therefore, (\ref{dimbm}) holds for all unconditional log-concave measures and unconditional convex sets, as well as for all even log-concave measures and symmetric convex sets in $\R^2$. 

\medskip

The main result of this paper is the following.

\begin{theorem}\label{main}
Let $\mu$ be a symmetric log-concave measure on $\R^n$ with density $e^{-V(x)},$ for some convex function $V:\R^n\rightarrow \R.$ Suppose that $k_1, k_2>0$ constants such that
\begin{equation}\label{k1}
\nabla^2 V\geq k_1 {\rm Id},
\end{equation} 
where $\nabla^2 V$ stands for the Hessian of $V,$ and
\begin{equation}\label{k2}
\Delta V\leq k_2 n.
\end{equation} 
Let $R={k_2}/{k_1} \ge 1$. Then for symmetric convex sets $K$ and $L,$ and any $\lambda\in[0,1]$, one has
\begin{equation}\label{GZ-proved}
\mu(\lambda K+(1-\lambda)L)^{\frac{c}{n}}\geq \lambda \mu(K)^{\frac{c}{n}}+(1-\lambda)\mu(L)^{\frac{c}{n}},
\end{equation}
where 
$$c=c(R)=\frac{2}{(\sqrt{R}+1)^2}.$$
\end{theorem}

Recall that the standard Gaussian measure $\gamma$ is the measure with the density $\left({1}/{\sqrt{2\pi}}\right)^ne^{-\frac{|x|^2}{2}}$. In this case, $\nabla V=x,$ $\nabla^2V={\rm Id}$, and hence $k_1=k_2=R=1$. Therefore, Theorem \ref{main} implies $1/2n-$concavity of the standard Gaussian measure. We shall prove a more general fact.

\begin{theorem}\label{corgauss}
Let $\gamma$ be the standard Gaussian measure. For  convex sets $K$ and $L$ in $\R^n$ which contain the origin, and any $\lambda\in[0,1]$, one has
\begin{equation}\label{OK}
\gamma(\lambda K+(1-\lambda)L)^{\frac{1}{2n}}\geq \lambda \gamma(K)^{\frac{1}{2n}}+(1-\lambda)\gamma(L)^{\frac{1}{2n}}.
\end{equation}
\end{theorem}

Interestingly, it was shown by Nayar and Tkocz \cite{NaTk} that only under the assumption of the sets containing the origin,  (\ref{GZ0}) fails in dimension two. Theorem \ref{corgauss} shows, however, that (\ref{OK}) does hold, even under this  assumption. See Remark \ref{TN} for more discussion.


In order to derive all our results, we reduce the problem to its infinitesimal version following the approach of \cite{Col1}, \cite{CLM}, \cite{CL}, \cite{KM1}, \cite{KM2}, \cite{KolMil}, \cite{KolMilsupernew}. In particular, we use a Bochner-type identity obtained in \cite{KM1}.
The arguments are based on the application of the elliptic boundary value problem $Lu=F$ with Neumann boundary condition 
$u_{\nu}=f$. Our main result  corresponds to the simplest choice of $F$, namely $F=1$. However, we demonstrate that a choice 
of non-constant $F$ can lead to sharp estimates  (see Section 6). This is an important observation which we believe could be useful for further developments. In Section 6 we also prove that constant $c$ in (\ref{GZ-proved}) can be estimated by the parameter
$$
 \inf_{K}  \left[1-\frac{1}{n\mu(K)}\int_K \langle (\nabla^2 V+\frac{1}{n}\nabla V\otimes \nabla V)^{-1}\nabla V, \nabla V \rangle \,d\mu\right],
$$
where the infimum is taken over all symmetric convex sets.

This paper is organized as follows. In Section 2, we outline the high-level structure of the proof of Theorem \ref{main}, with the goal of indicating the main steps in the estimate. In Sections 3, 4 and 5 we proceed with the said steps, one at a time. At the end of Section 5 we include  the proof of 
Theorem \ref{corgauss}. In Section 6 we discuss some concluding remarks: namely, in subsection 6.1 we formulate a more general version of Theorem \ref{main} and in subsection 6.2 we discuss a more general approach to the proof which recovers the result of Gardner and Zvavitch about dilates of convex bodies.
\\
\\
\textbf{Acknowledgement.} First author supported by RFBR project 17-01-00662,  DFG project RO 1195/12-1 and the Simons Foundation.
Second author supported by NSF CAREER DMS-1753260. The article  was prepared within the framework of the HSE University Basic Research Program and funded by the Russian Academic Excellence Project '5-100'.  The work was partially supported by the National Science Foundation under Grant No. DMS-1440140 while the authors were in residence at the Mathematical Sciences Research Institute in Berkeley, California, during the Fall 2017 semester. The authors are grateful to Emanuel Milman for fruitful discussions. The authors are very grateful to Tomacz Tkocz for pointing out to them Lemma \ref{nowGauss}, leading to the formulation of Theorem \ref{corgauss}. The authors are grateful to the anonymous referee for the detailed and helpful report.

\section{High-level structure of the proof}

We shall work in $\R^n.$ Throughout, $K$ stands for a convex body (compact convex set with non-empty interior) and $\mu$ for a log-concave measure with density $e^{-V}$,
where $V$ is convex function. 
The norm sign $||\cdot||$ with respect to a matrix stands for the Hilbert-Schmidt norm 
$$
\|A\|=\sqrt{\rm{Tr} (A A^{T})}.
$$ 
Given  vectors $a,b\in \R^n$ the corresponding tensor product $a\otimes b$ is a bilinear form defined  by
$$
a\otimes b (v, w) = \langle a, v \rangle \langle b , w \rangle.
$$
We shall assume without loss of generality that $V$ is twice continuously differentiable, the boundary of $K$ is $C^2$-smooth and $K$ is strictly convex; the general bounds follow by approximation.
The notation  $\nabla^2 u$ stands for the Hessian matrix of $u$.


Most of our results deal with the following two classes of sets which are closed under Minkowski convex combinations:
$$
\mathcal{F}_{sym} = \{ {\rm symmetric \ convex \ sets} \},
$$
$$
\mathcal{F}_{o} = \{ {\rm  convex \ sets \ containing \ the  \ origin} \}.
$$

In this section, we outline the steps of the proof by gradually introducing several definitions and lemmas. Proofs of the lemmas will be given in subsequent sections.

\begin{definition}\label{c}
Fix the dimension $n\in\N.$ Consider a family $\mathcal{F}$ of convex sets in $\R^n$ which is closed under Minkowski convex combinations. The Gardner-Zvavitch constant $C_0= C_0(\mu, \mathcal{F})$ is the largest number so that for all convex sets $K, L \in \mathcal{F}$, and for any $\lambda\in[0,1]$,
\begin{equation}\label{GZdef}
\mu(\lambda K+(1-\lambda)L)^{\frac{C_0}{n}}\geq \lambda \mu(K)^{\frac{C_0}{n}}+(1-\lambda)\mu(L)^{\frac{C_0}{n}}.
\end{equation}
\end{definition}

It can be verified, by considering small balls centered at the origin, that   
$$C_0(\mu, \mathcal{F}_{sym}) \leq 1$$ 
for every log-concave measure $\mu$ which is not supported on a proper subspace of $\R^n$. By H\"older's inequality, (\ref{GZdef}) implies (\ref{GZ0}) for all $c\in [0, C_0]$. Therefore, we shall be concerned with estimating $C_0$ from below.


We consider the weighted Laplace operator $L$ associated with the measure $\mu$, that is 
\begin{equation}\label{L}
Lu=\Delta u-\langle \nabla u, \nabla V\rangle.
\end{equation} 
In the case when $\mu$ is Gaussian, this operator is commonly referred to as the Ornstein-Uhlenbeck operator. We shall make use of the generalized integration by parts identity: for any $u,v\in C^2(\R^n)$,
$$\int_{\R^n} v\cdot Lu \,d\mu=-\int_{\R^n} \langle \nabla v, \nabla u\rangle \,d\mu.$$


\begin{definition}\label{C_1}
Define $C_1=C_1(\mu,\mathcal{F})$ to be the largest number, such that for every $u\in C^2(K)$ and $K \in \mathcal{F}$ with $Lu=1_K,$ 
$$\frac{1}{\mu(K)}\int_K ||\nabla^2 u||^2+\langle \nabla^2V \nabla u, \nabla u\rangle \,d\mu\geq \frac{C_1(\mu,\mathcal{F})}{n}.$$
\end{definition}

The first key step in our proof is outlined in the following lemma:

\begin{lemma}\label{lemma1}
For every family $\mathcal{F}$ of convex sets in $\R^n$ closed under Minkowski convex combinations,
$$
C_0(\mu,\mathcal{F})\geq C_1(\mu, \mathcal{F}).
$$
\end{lemma}

Next, we conclude with two more lemmas.

\begin{lemma}\label{lemma2}
Assume that   $\nabla^2 V\geq k_1 {\rm Id}$.

\begin{enumerate}
\item
Assume, in addition, that $V$ is even.
 Then for every $\varepsilon\in[0,1]$,
$$C_1(\mu,\mathcal{F}_{sym})\geq \frac{1}{\mu(K)}\int_K \frac{1}{\frac{|\nabla V|^2}{(1+\varepsilon)nk_1}+\frac{1}{1-\varepsilon}} \,d\mu.$$

\item
For every family $\mathcal{F}$ of convex sets which is closed under Minkowski convex combinations, one has
$$C_1(\mu,\mathcal{F})\geq \frac{1}{\mu(K)}\int_K \frac{1}{\frac{|\nabla V|^2}{nk_1}+1} \,d\mu.$$
\end{enumerate}

\end{lemma}

\begin{lemma}\label{lemma3}
Fix a convex function $V$ on $\R^n$. Assume that $\Delta V\leq k_2 n$. Fix a constant $k_1>0$ and let $R={k_2}/{k_1}$.
\begin{enumerate}
\item If a convex set $K$ and the measure $\mu$ with density $e^{-V}$ satisfy $\int_{K} \nabla V d \mu=0$, then there exists an $0 < \varepsilon <1$ such that
$$\frac{1}{\mu(K)}\int_K \frac{1}{\frac{|\nabla V|^2}{(1+\varepsilon)nk_1}+\frac{1}{1-\varepsilon}} \,d\mu\geq \frac{2}{(\sqrt{R}+1)^2}.$$
\item For the standard Gaussian measure $\gamma$ and for every convex set $K$ which contains the origin, we have
$$\frac{1}{\gamma(K)}\int_K \frac{1}{\frac{|x|^2}{n}+1} \,d\gamma \geq \frac{1}{2}.$$
\end{enumerate}
\end{lemma}

\textbf{Proof of Theorem \ref{main}.} The theorem follows immediately from Lemma \ref{lemma1} applied to $\mathcal{F}=\mathcal{F}_{sym}$, Lemma \ref{lemma2} (1) and Lemma \ref{lemma3} (1), in view of the Definition \ref{c}. $\square$ 

\medskip

\textbf{Proof of Theorem \ref{corgauss}.} Since $\nabla V=x$ and $k_1=1,$ the theorem follows immediately from Lemma \ref{lemma1}, Lemma \ref{lemma2} (2) and Lemma \ref{lemma3} (2), in view of the Definition \ref{c}.
$\square$


\section{Proof of Lemma \ref{lemma1}}

The proof of Lemma \ref{lemma1} is a combination of a variational argument, integration by parts, and an application of the Cauchy-Schwarz inequality. We start by introducing the variational argument.

\subsection{Variational argument}
Infinitesimal versions of Brunn-Minkowski type inequalities have been considered and extensively studied in \cite{BakLed}, \cite{BL1}, \cite{BL2}, \cite{Col1}, \cite{Col2}, \cite{Colesanti-Hug-Saorin2}, \cite{KolMil-1}, \cite{KolMil}, \cite{KolMilsupernew}, \cite{CLM}, \cite{CL}.

Following Schneider (\cite{book4}, page 115)  we say that a convex body $K$ is of class $C^2$ if its support function is of class $C^2.$ Further, we say that $K$ is of class $C^{2}_{+}$ if $K$ is of class $C^2$ and admits positive Gauss curvature. We say that a function $h \colon \sfe \to \mathbb{R}$ is a $C^{2}_{+}(\sfe)$-function if it is a support function of a $C^{2}_{+}$ convex body.


Let $h$ be the support function of a $C^{2}_{+}$ convex body $K$ and let $\psi\in C^2(\sfe)$. Then
\begin{equation}\label{additive perturbation}
h_s=h+s\psi\in C^{2}_{+}(\sfe),
\end{equation}
if $s$ is sufficiently small (say $|s|\leq a$ for some appropriate $a>0$). Hence for every $s$ in this range there 
exists a unique $C^{2}_{+}$ convex body $K_s$ with support function $h_s$. 
For an interval $I$, we define the one-parameter family of convex bodies
$$
\mathcal{K}(h,\psi,I)=\{K_s\,:\, h_{K_s}=h+s\psi,\, s\in I\}.
$$

\begin{lemma}\label{key_lemma_BM}
Assume that $\mu$ is a  log-concave measure with twice continuously differentiable density, $c$ is a positive constant, and $\mathcal{F}$ is a family of convex sets closed under Minkowski convex combinations. 
The inequality 
\begin{equation}\label{GZc}
\mu(\lambda K+(1-\lambda)L)^{\frac{c}{n}}\geq \lambda \mu(K)^{\frac{c}{n}}+(1-\lambda)\mu(L)^{\frac{c}{n}}
\end{equation} 
 holds for all $K, L \in \mathcal{F}$ and every $\lambda\in [0,1]$, if and only if for every one-parameter family $\mathcal{K}(h,\psi,I)$ such that $K_s \in \mathcal{F}$ for every $s\in I$, one has
\begin{equation}\label{dimbmeq}
\left. \frac{d^2}{ds^2} \mu(K_s) \right|_{s=0}\cdot \mu(K_0)\le \frac{n-c}{n}\left(\left. \frac{d}{ds} \mu(K_s)\right|_{s=0}\right)^2.
\end{equation}
\end{lemma}
\begin{proof}
Assume first that $\mu$ satisfies \eqref{GZc}. The equality $h_{K_s}=h+s\psi$, $s\in I$, and the linearity of support functions with respect to Minkowski addition, imply that for every $s,t\in I$ and for every $\lambda\in[0,1]$
$$
K_{\lambda s+(1-\lambda)t}=\lambda K_s+(1-\lambda)K_t.
$$
The inequality (\ref{GZc}) implies
$$
\mu(K_{\lambda s+(1-\lambda)t})^{\frac{c}{n}}=\mu(\lambda K_s+(1-\lambda)K_t)^{\frac{c}{n}}\geq 
\lambda \mu(K_s)^{\frac{c}{n}}+(1-\lambda)\mu(K_t)^{\frac{c}{n}},
$$
which means that the function $\mu(K_s)^{\frac{c}{n}}$ is concave in $s$ on $I$. This implies (\ref{dimbmeq}). We add that $\mu(K_s)$ is twice differentiable in $s$, in view of our smoothness assumptions on $K$ and $\psi$. This fact can be observed, for example, using Lemma 6.1 from \cite{CLM}, in which $\mu(K)$ is expressed in terms of the support function of $K$. Plugging $h+s\psi$ into this formula gives a twice differentiable function in $s$. In the notation of Lemma 6.1 in \cite{CLM}, $F$ stands for the density of $\mu$.

Conversely, suppose that for every system $\mathcal{K}(h,\psi,I)$ with $K_s \in \mathcal{F}$, whenever $s\in I,$ the function $\mu(K_s)^{\frac{c}{n}}$ has 
non-positive second derivative at $0$, i.e. \eqref{dimbmeq} holds. We observe that this implies concavity of $\mu(K_s)^{\frac{c}{n}}$ on 
the entire interval $I$. Indeed, given $s_0$ in the interior of $I$, 
consider $\tilde{h}=h+s_0\psi$, and define a new system $\tilde{\mathcal{K}}(\tilde{h},\psi,J)$, where $J$ is a new interval such that
$\tilde{h}+s\psi=h+(s+s_0)\psi\in C^{2}_{+}$ for every $s\in J$. Then the second derivative of
$\mu(K_s)^{\frac{c}{n}}$ at $s=s_0$ is negative, as it is equal to the second derivative of $\mu(\tilde{K}_s)^{\frac{c}{n}}$ at $s=0$.
Thus (\ref{dimbmeq}) implies  concavity of $s \to \mu(K_s)^{\frac{c}{n}}$ on $[0,1]$ :
$$
\mu^{\frac{c}{n}}(K_s) \ge s  \mu^{\frac{c}{n}}(K_1) + (1-s) \mu^{\frac{c}{n}}(K_0), \ \forall s \in [0,1].
$$
Take $s=1-\lambda$,   $h=h_K$, $\psi=h_L-h_K$ and observe that $K_s = \lambda K + (1-\lambda) L$. Therefore (\ref{GZc}) holds. This completes the proof.
\end{proof}

The normal vector to the boundary of $K$ at the point $x$ will be denoted by $n_x$. Recall our assumption that $K$ is strictly convex and $C^2-$smooth, so the outward unit normal vector is unique; the general case may be derived by approximation.  We shall write 
$$ \mu_{\partial K} (x)=e^{-V(x)} \cdot \mathcal{H}^{n-1}|_{\partial K},
$$ 
where $\mathcal{H}^{n-1}$ stands for the $(n-1)$-dimensional Hausdorff measure; the notation $\nabla_{\partial K}$ means the boundary gradient (i.e., the projection of the gradient onto the support hyperplane). 
The second fundamental form of $\partial K$ will be denoted by $\rm{II}$, and the weighted mean curvature at a point $x$ is given by
$$H_x= tr({\rm II})-\langle \nabla V, n_x\rangle.$$

The following proposition was shown by  Kolesnikov and Milman \cite{KM2} (see the proof of Theorem 6.6.):
\begin{proposition}\label{derivatives}
Let $f:\partial K\rightarrow \R$ be given by $f(x)=\psi(n_x)$. Then
$$\mu(K_s)'|_{s=0}=\int_{\partial K}  f(x)  \,d\mu_{\partial K} (x);$$
$$\mu(K_s)''|_{s=0}=\int_{\partial K}  \left(H_x f^2-\langle \mbox{\rm{II}}^{-1}\nabla_{\partial K}  f,  \nabla_{\partial K}  f\rangle \right) \,d\mu_{\partial K} (x).$$
\end{proposition}

\begin{definition}
For a fixed class $\mathcal{F}$ of convex sets which is closed under dilates, and a convex body $K\in \mathcal{F}$, we consider a class $C_{\mathcal{F}}(K)$ of $C^2$-smooth real-valued functions on $\partial K$, given by
$$C_{\mathcal{F}}(K)=\{f(x)=h_L(n_x)-h_K(n_x):\,L\in\mathcal{F},\,t>0\}\cap C^2(\partial K).$$
\end{definition}

Lemma \ref{key_lemma_BM} and Proposition \ref{derivatives} imply:
\begin{cor}\label{equiv-const}
Fix a class $\mathcal{F}$ of convex sets in $\R^n$ which is closed under Minkowski convex combinations. Suppose that for any convex body $K \in \mathcal{F}$ and for any function $f(x)\in C_{\mathcal{F}}(K)$, 
\begin{equation}\label{infin-1}
\int_{\partial K}  \left(H_x f^2-\langle \mbox{\rm{II}}^{-1}\nabla_{\partial K}  f,  \nabla_{\partial K}  f\rangle\right)  \,d\mu_{\partial K} (x)-
\frac{n-C}{n\mu(K)}\left(\int_{\partial K}  f(x)  \,d\mu_{\partial K} (x)\right)^2\leq 0.
\end{equation}
Then $$C_0(\mu,\mathcal{F})\geq C.$$
\end{cor}

\subsection{Integration by parts}
The following Bochner-type identity was obtained by Kolesnikov and Milman. It is a particular case of Theorem 1.1 in \cite{KM1}
(note that $\rm{Ric}_{\mu}  = \nabla^2 \it{V}$ in our case).  This is a generalization of a classical result of R.C.~Reilly.

\begin{proposition}\label{raileyprop}
Let $u \in C^2(K)$  and $u_{n} = \langle \nabla u, n_x \rangle \in C^1(\partial K)$. Then  
\begin{align}
\label{railey}
\int_{K} (L u)^2 d \mu & =\int_{K} \left(||\nabla^2 u||^2+\langle\nabla^2 V \nabla u, \nabla u\rangle\right) d \mu+
\\&  \nonumber\int_{\partial K}  (H_x u_n^2 -2\langle \nabla_{\partial K} u, \nabla_{\partial K}  u_n\rangle +\langle \mbox{\rm{II}} \nabla_{\partial K} u, \nabla_{\partial K}  u\rangle )  \,d\mu_{\partial K} (x).
\end{align}
\end{proposition}

\subsection{Proof of Lemma \ref{lemma1}.}

In view of  Corollary \ref{equiv-const}  it is sufficient to verify (\ref{infin-1}) with $C=C_1(\mu,\mathcal{F})$.
Fix a $C^1$ function $f:\partial K\rightarrow \R$. 
 In the case when $\int_{\partial K} f \,d\mu_{\partial K}=0$, we automatically get (\ref{infin-1}) with an arbitrary constant $C$, as a consequence of the log-concavity of $\mu$ (see Theorem 1.1 in \cite{KM2}). Indeed, in this case (\ref{infin-1}) is simply identical to the infinitesimal form of log-concavity.
 

If $\int_{\partial K} f \,d\mu_{\partial K} \ne 0$, then after a suitable renormalization one can assume that $\int_{\partial K} f \,d\mu_{\partial K} =\mu(K)$.
 
Let $u$ be the solution of the Poisson equation
 $$Lu=1$$
with the Neumann boundary condition for every $x\in\partial K$
$$\langle \nabla u(x), n_x\rangle = f(x).$$  
We refer to subsection 2.4  in \cite{KM2}, where the reader can find the precise statement ensuring well-posedness of this equation and several references
to classical PDE's textbooks for further reading.  


Applying (\ref{railey}) and the definition of $C_1(\mu, \mathcal{F})$ one obtains
$$
\mu(K)\ge \frac{C_1(\mu,\mathcal{F})}{n} \mu(K)  + \int_{\partial K}  (H_x f^2 -2\langle \nabla_{\partial K} u, \nabla_{\partial K}  f\rangle +\langle \mbox{\rm{II}} \nabla_{\partial K} u, \nabla_{\partial K}  u\rangle )  \,d\mu_{\partial K} (x).
$$






Recall that for a symmetric positive-definite matrix $A$,
\begin{equation}\label{matr}
\langle Ax,x\rangle+\langle A^{-1} y, y\rangle \geq 2\langle x,y\rangle.
\end{equation}
Indeed, choosing an orthogonal frame making $A$ diagonal with eigenvalues $\lambda_i$ we reduce (\ref{matr})
to the  inequality 
$$\sum_{i=1}^n \lambda_i x^2_i + \sum_{i=1}^n {y^2_i}/{\lambda_i} \ge 2 \sum_{i=1}^n x_i y_i,$$ 
which follows from the arithmetic-geometric mean and Cauchy-Schwarz inequalities. 

Applying (\ref{matr}) with $A=\rm{II}$, $x=\nabla_{\partial K} u$ and $y=\nabla_{\partial K} f,$ we obtain
$$
\int_{\partial K}  \left(H_x f^2-\langle \mbox{\rm{II}}^{-1}\nabla_{\partial K}  f,  \nabla_{\partial K}  f\rangle\right)  \,d\mu_{\partial K} (x)-
\frac{n-C_1(\mu,\mathcal{F})}{n} \mu(K)\leq 0.$$
The result of the lemma now follows from Corollary \ref{equiv-const}. $\square$




\section{Proof of Lemma \ref{lemma2}}

Firstly, suppose that $u$ is a $C^2$-smooth function on a symmetric convex set $K$ with $Lu=1_K$ on $K$. 

Since $K$ is symmetric and $V$ is even, the function $u$ is even as well.  Indeed, we get by symmetry that $(u(x) + u(-x))/{2}$ is a solution to our system as well. Uniqueness of the solution implies $u(-x)=u(x)$.

To prove the lemma, it suffices to show that 
\begin{equation}\label{toshow}
\int_{K} ||\nabla^2 u||^2+\langle\nabla^2 V \nabla u, \nabla u\rangle\, d \mu\geq \frac{1}{n}\int_K \frac{1}{\frac{|\nabla V|^2}{(1+\varepsilon)nk_1}+\frac{1}{1-\varepsilon}} \,d\mu.
\end{equation}

By the Cauchy-Schwarz inequality,
\begin{equation}\label{hessian}
\int_{K} ||\nabla^2 u||^2 d \mu \geq \frac{1}{n} \int_{K} |\Delta u|^2 d \mu.
\end{equation}
Indeed, to see why (\ref{hessian}) is true, recall that $||\nabla^2 u||^2=\sum_{i=1}^n \lambda^2_i$, where $\lambda_1,...,\lambda_n$ are the eigenvalues of $\nabla^2 u$, and recall also that $\Delta u=\sum_{i=1}^n \lambda_i$.

Note that the symmetry of $u$ implies
\begin{equation}\label{symmetry1}
\int_{K} u_{x_i} d \mu=0.
\end{equation}
By the Brascamp--Lieb inequality (see \cite{BGL}, Theorem 4.9.1, or \cite{BrLi}), we have
$$
\int_{K} u^2_{x_i} d \mu \le \int_{K} \langle (\nabla^2 V)^{-1} \nabla u_{x_i}, \nabla u_{x_i} \rangle d \mu.
$$
Applying the lower bound for $\nabla^2 V$ and summing over $i=1,...,n$, we get
\begin{equation}\label{CFM+hess}
\int_{K} ||\nabla^2 u||^2 d \mu \geq  k_1\int_{K} |\nabla u|^2  \,d\mu.
\end{equation}
In addition, we observe that the lower bound $\nabla^2 V\geq k_1 {\rm Id}$ also yields
\begin{equation}\label{addition}
\int_{K} \langle\nabla^2 V \nabla u, \nabla u\rangle d \mu \geq  k_1\int_{K} |\nabla u|^2 d \mu.
\end{equation}
Let $\varepsilon>0$. Multiplying (\ref{hessian}) by $1-\varepsilon$, multiplying (\ref{CFM+hess}) by $\varepsilon$, summing the resulting two inequalities, and then using (\ref{addition}), we arrive at
\begin{equation}\label{eq2}
\int_{K} (||\nabla^2 u||^2+\langle\nabla^2 V \nabla u, \nabla u\rangle) d \mu\geq \int_{K} \Bigl( \frac{1-\varepsilon}{n} |\Delta u|^2 +k_1(1+\varepsilon) |\nabla u|^2  \Bigr) d \mu. 
\end{equation}
Writing 
$$\Delta u=Lu+\langle \nabla V, \nabla u\rangle=1_K+\langle \nabla V, \nabla u\rangle,$$ 
we get that the right hand side of (\ref{eq2}) equals
\begin{equation}\label{eq3}
\int_{K}  \Bigl[ \frac{1-\varepsilon}{n}  1+2\langle \nabla u, \frac{1-\varepsilon}{n}\nabla V\rangle+\langle A_{\varepsilon} \nabla u, \nabla u\rangle \Bigr] d \mu, 
\end{equation}
where
$$A_{\varepsilon}=\frac{1-\varepsilon}{n}\nabla V\otimes\nabla V+k_1(1+\varepsilon) {\rm Id}.$$
Note that $A_{\varepsilon}$ is positive semi-definite, since it is a sum of positive semi-definite matrices. Using (\ref{matr}) once again, this time with $A=A_{\varepsilon}$, $x=\nabla u$ and $$y=-\frac{1-\varepsilon}{n}  \nabla V,$$ we see that (\ref{eq3}) is greater than or equal to
 
\begin{equation}\label{eq4}
\int_{K} \frac{1-\varepsilon}{n}  \left(1-\frac{1-\varepsilon}{n}\langle A_{\varepsilon}^{-1}\nabla V, \nabla V\rangle\right) d \mu.
\end{equation} 

We observe that for any vector $z\in\R^n$ and for all $a, b\in \R,$
\begin{equation}\label{eq5}
\left( a {\rm Id}+b z\otimes z\right)^{-1}z=\frac{z}{a+b|z|^2}.
\end{equation}
 
Applying (\ref{eq5}) with $a=(1-\varepsilon)/n$, $b=k_1(1+\varepsilon),$ and $z=\nabla V$, we rewrite (\ref{eq4}) as 

\begin{equation}\label{eq6}
k_1(1+\varepsilon)\int_{K} \frac{1}{|\nabla V|^2+k_1 n\frac{1+\varepsilon}{1-\varepsilon}} d \mu
=  \int_{K} \frac{d \mu}{\frac{|\nabla V|^2}{k_1(1+\varepsilon)}+ \frac{n}{1-\varepsilon}}.
\end{equation}  
The proof of part (1) is complete.

Secondly, if the class $\mathcal{F}$ is arbitrary, we apply the same estimate with $\varepsilon=0$ and avoid using (\ref{CFM+hess}). Note that (\ref{CFM+hess}) is the only place where the symmetry was used. This completes the proof of part (2).
$\square$





\section{Proof of Lemma \ref{lemma3}.}

We shall need the following lemma, where symmetry is used in the crucial way: namely, we use the simple fact that log-concave \emph{even} functions on the real line are concave at zero.

\begin{lemma}\label{k_2}
For a log-concave measure $\mu$ with density $e^{-V}$ and a convex body $K$, satisfying
\begin{equation}\label{condnosym}
\int_K \frac{\partial V}{\partial x_i} \,d\mu=0, 
\end{equation}
for all $i=1,...,n$, we have
$$\int_{K} |\nabla V(x)|^2 d \mu \leq \int_{K} \Delta V d \mu.$$
\end{lemma}
\begin{proof} Let $i\in\{1,...,n\}$. By  the Pr\'ekopa-Leindler inequality (\cite{Gar}, Theorem 4.2), the function
$$g(t)=\int_K e^{-V(x+te_i)} \,dx$$
is log-concave in $t.$ In particular,
\begin{equation}\label{g-logconc}
g(0)g''(0)-g'(0)^2\leq 0.
\end{equation}
Note that 
\begin{equation}\label{symmetry2}
g'(0)=-\int_K \frac{\partial V}{\partial x_i} e^{-V(x)} \,dx=0.
\end{equation}
Therefore, by (\ref{g-logconc}),
\begin{equation}\label{g''}
g''(0)=\int_K \left(-\frac{\partial^2 V}{\partial^2 x_i}+\left(\frac{\partial V}{\partial x_i}\right)^2 \right) e^{-V(x)} \,dx\leq 0.
\end{equation}
Applying (\ref{g''}) and summing over $i=1,...,n$, we obtain the conclusion of the lemma.
\end{proof}

\begin{remark}
Alternatively, Lemma \ref{k_2} follows directly for the Brascamb--Lieb inequality applied to the functions $V_{x_i}$:
$$
\int_{K} V^2_{x_i}d \mu \le \int \langle (D^2 V)^{-1} \nabla V_{x_i}, \nabla V_{x_i} \rangle d \mu = \int_{K} V_{x_i x_i} d \mu.
$$  
Here we use log-concavity of the measure $1_K e^{-V} dx$
\end{remark}

The next lemma shows that, in the case of the standard Gaussian measure, the conclusion of Lemma \ref{k_2} holds under an even weaker assumption of the sets containing the origin. Recall that a set $K$ is called star-shaped if it contains the interval $\{ t x, t \in [0,1] \}$ for every $x \in K$.

\begin{lemma}\label{nowGauss}
Suppose $K$ is a star-shaped body, and $\gamma$ is the standard Gaussian measure. Then
\begin{equation}\label{eq-key-new}
\int_K |x|^2 d\gamma(x)\leq n\gamma(K).
\end{equation}
\end{lemma}
\begin{proof} Consider the function $g(s)=\gamma(sK)$. Note that $g$ is non-decreasing, since $K$ is star-shaped. Observe that, by Proposition \ref{derivatives},
$$
g'(1)=\frac{1}{(2 \pi)^{\frac{n}{2}}} \int_{\partial K} \langle n_x, x \rangle e^{-\frac{x^2}{2}} d \mathcal{H}^{n-1},
$$
where by $d \mathcal{H}^{n-1}$ we denote the Hausdorff measure on $\partial K.$

Applying the divergence theorem, we therefore get
$$
0 \le g'(1) = \int_{K} {\rm{div}} \Bigl( \frac{1}{(2 \pi)^{\frac{n}{2}}} x e^{-\frac{|x|^2}{2}} \Bigr) dx = n \gamma(K) -  \int_{K} |x|^2 d \gamma.
$$ This inequality implies  (\ref{eq-key-new}).
\end{proof}

\subsection{Proof of Lemma \ref{lemma3}.}
To prove (1)  we use Jensen's inequality (\cite{Rudin}, Theorem 3.3) and convexity of the function $x\rightarrow{1}/{(1+x)}$ for $x>0$. We get
\begin{equation}\label{toref}
\frac{1}{\mu(K)}\int_{K} \frac{1}{\frac{|\nabla V|^2}{(1+\varepsilon)nk_1}+\frac{1}{1-\varepsilon}} \ \,d\mu \geq
 \frac{1}{ \frac{1}{\mu(K)}\int_{K} \frac{|\nabla V|^2}{(1+\varepsilon)nk_1} d \mu+\frac{1}{1-\varepsilon}}.
\end{equation}
Next, we apply (\ref{toref}) and Lemma \ref{k_2} along with the assumption $\Delta V\leq n k_2,$ to infer that
\begin{equation}\label{eq32}
\frac{1}{\mu(K)}\int_K \frac{1}{\frac{|\nabla V|^2}{(1+\varepsilon)nk_1}+\frac{1}{1-\varepsilon}} \,d\mu\geq \frac{1}{\frac{R}{1+\varepsilon}+\frac{1}{1-\varepsilon}},
\end{equation}
where, as before, $R={k_2}/{k_1}$. 
Plugging in the optimal value of
$$\varepsilon=\frac{R+1-2\sqrt{R}}{R-1},$$
we finish the proof of part (1). 

Next, to obtain part (2) of the Lemma, we substitute $\varepsilon =0$ in (\ref{toref}) to arrive at
\begin{equation}\label{benjerk}
\frac{1}{\mu(K)}\int_{K} \frac{1}{\frac{|\nabla V|^2}{nk_1}+1} \ \,d\mu \geq
 \frac{1}{ \frac{1}{\mu(K)}\int_{K} \frac{|\nabla V|^2}{nk_1} d \mu+1}.
\end{equation}
Recalling that $V(x)=|x|^2/2$, part (2) follows from applying Lemma \ref{nowGauss} to the right hand side of (\ref{benjerk}). $\square$

\medskip
Note, that in the case of the standard Gaussian measure the optimal choice is $\varepsilon=0.$


\section{Concluding remarks}

\subsection{An improved estimate.}

Throughout this subsection sets are assumed to be origin-symmetric and functions are assumed to be even. 

We outline a sharper, more general estimate for the Gardner-Zvavitch constant.

We recall that that $C(K,\mu)$ is called the Poincar{\'e} constant of $\mu|_{K}$ if it is the smallest number $a$ such that for all $C^1-$smooth functions $f$ on $K$, one has
\begin{equation}\label{eq-poin}
\int_{K} f^2 d\mu-\frac{1}{\mu(K)}\left(\int_K f d\mu\right)^2\leq a\int_K |\nabla f|^2 d\mu.
\end{equation}

\begin{theorem}\label{improved}
Let $\mathcal{F}$ be a collection of origin-symmetric convex bodies in $\R^n$ which is closed under Minkowski convex combinations.   Let
$$C=C(\mu, \mathcal{F})=\sup_{\varepsilon \in [0,1)} (1 - \varepsilon)  \inf_{K\in\mathcal{F}} \left[1-\frac{1}{n\mu(K)}\int_K \langle A^{-1}\nabla V, \nabla V \rangle \,d\mu\right],$$
where 
 $$A=   \nabla^2 V+\frac{1}{n}\nabla V\otimes\nabla V+ \frac{\varepsilon}{(1-\varepsilon)C(K,\mu)} {\rm Id}$$
 and $C(K,\mu)$ is the Poincar{\'e} constant of $\mu|_{K}$.
  
Then, for all $K, L\in \mathcal{F},$ and for every $\lambda\in [0,1]$
$$\mu(\lambda K+(1-\lambda)L)^{\frac{C}{n}}\geq \lambda \mu(K)^{\frac{C}{n}}+(1-\lambda)\mu(L)^{\frac{C}{n}}.$$

In particular, 
$$
C \ge \inf_{K\in\mathcal{F}}  \left[1-\frac{1}{n\mu(K)}\int_K \langle (\nabla^2 V+\frac{1}{n}\nabla V\otimes \nabla V)^{-1}\nabla V, \nabla V \rangle \,d\mu\right].
$$

\end{theorem}
\begin{proof} Consider an even $C^2$ function $u:K\rightarrow\R$ such that $Lu=1_K.$ Then, by (\ref{hessian}), along with the fact that $\Delta u=1+\langle \nabla V, \nabla u\rangle$,
\begin{align*}
\int_K ||\nabla^2 u||^2 & +\langle \nabla^2 V\nabla u,\nabla u\rangle \,d\mu\geq
\int_K \frac{1}{n} |1+\langle \nabla V, \nabla u\rangle|^2+\langle \nabla^2 V\nabla u,\nabla u\rangle \,d\mu
\\& = \int_K \frac{1}{n} +\frac{2}{n}\langle \nabla V, \nabla u\rangle+\langle (\nabla^2 V+\frac{1}{n}\nabla V\otimes\nabla V)\nabla u,\nabla u\rangle \,d\mu.
\end{align*}
Next we apply the Poincar\`e inequality (\ref{eq-poin}) to every $u_{x_i}$ (here we use that $u$ and $\it{V}$ are even, hence $\int u_{x_i} d\mu=0$) :
$$
\int_{K} u^2_{x_i} d \mu  \le C(K,\mu) \int_{K} |\nabla u_{x_i}|^2 d \mu.
$$
Thus 
$$
\int_{K} |\nabla u|^2 d \mu  \le  C(K,\mu) \int_{K} \| \nabla^2 u \|^2 d \mu,
$$
and for every $\varepsilon \in [0,1]$ one has
\begin{align*}
\int_K ||\nabla^2 u||^2 & +\langle \nabla^2 V\nabla u,\nabla u\rangle \,d\mu\geq
\frac{\varepsilon}{C(K,\mu)} \int_{K} |\nabla u|^2 d \mu  \\&
+ (1-\varepsilon) \int_K \frac{1}{n} +2\frac{\langle \nabla V, \nabla u\rangle}{n}+\langle (\nabla^2 V+\frac{1}{n}\nabla V\otimes\nabla V)\nabla u,\nabla u\rangle \,d\mu
\\& =  (1-\varepsilon) \Bigl( \int_K \frac{1}{n} +2\frac{\langle \nabla V, \nabla u\rangle}{n} +
\langle  A \nabla u,\nabla u\rangle \,d\mu\Bigr).
\end{align*}
Applying (\ref{matr}) with the positive-definite matrix $A$, and Lemma \ref{lemma1}, we complete the proof.
\end{proof}

Theorem \ref{main} follows directly from Theorem \ref{improved}. Perhaps, $C(\mu,\mathcal{F})$ could be estimated for the class of symmetric convex sets under less restrictive assumptions than $\nabla^2 V\geq k_1 {\rm Id}$ and $\Delta V\leq n.$ 

\subsection{The case of non-constant $F$, and the Gardner-Zvavitch conjecture for dilates.}

\medskip

In this subsection we show that the choice of a constant $F$ in the equation $Lu=F$ is not always optimal.
We give an example showing that a result could be obtained with a non-constant $F$. 

\begin{definition}\label{C_2}
For a $C^2-$smooth \emph{even} function $F:K\rightarrow \R$, with $\int_K F\,d\mu\neq 0,$ let $C_F$ be the largest number, such that for every $u\in C^2(K)$ with $Lu=F,$ 
\begin{equation}\label{maineq-end-def}
\int_K ||\nabla^2 u||^2+\langle \nabla^2V \nabla u, \nabla u\rangle \,d\mu \geq \int_K F^2\,d\mu-\frac{n-C_F}{n\mu(K)}\left(\int_K F\,d\mu\right)^2.
\end{equation}
We define 
$$C_2(\mu)=\sup_{F} C_F,$$
where the supremum runs over all $C^2-$smooth \emph{even} functions $F:K\rightarrow \R$, with $\int_K F\,d\mu\neq 0.$
\end{definition}

We observe the following straightforward 

\begin{Claim}\label{claim}
$C_2(\mu)\geq C_1(\mu,\mathcal{F}_{sym}).$
\end{Claim}

Note that the proof of Lemma \ref{lemma1} implies, in fact, a stronger statement:

\begin{lemma}\label{claim2}
$C_0(\mu, \mathcal{F}_{sym})\geq C_2(\mu).$
\end{lemma}

It is possible that in the case of the standard Gaussian measure, the only sub-optimal place in our argument is the application of Lemma \ref{lemma1} in place of the stronger statement of Lemma \ref{claim2}: indeed, solving the Neumann system with $F\neq 1_K$ could lead to a better bound, however our current proof of Lemma \ref{lemma2} does not allow us to use this freedom.

Finally, we outline the following result.

\begin{lemma}\label{cl2}
Let $K$ be a convex body with $\int_K x d\gamma(x)=0$, let $\gamma$ be the Gaussian measure and let 
$$V(x)=u(x)=\frac{|x|^2}{2}$$ 
on $K$. Let 
$$F=Lu=n-|x|^2$$ 
on $K.$ Then
\begin{equation}\label{maineq-end}
\int_K ||\nabla^2 u||^2+\langle \nabla^2V \nabla u, \nabla u\rangle \,d\gamma \geq \int_K F^2\,d\gamma-\frac{n-1}{n\gamma(K)}\left(\int_K F\,d\gamma\right)^2.
\end{equation}
\end{lemma}
\begin{proof}

For all $x\in K,$
$$\frac{1}{4} ||\nabla^2 {|x|^2} ||^2=n;\,\, \frac{1}{4} \Bigl|\nabla {|x|^2}\Bigr|^2=|x|^2.$$
Hence, (\ref{maineq-end}) becomes
\begin{align}\label{1}
n \gamma(K)& +\int_K |x|^2 \,d\gamma\geq n^2 \gamma(K) -2n\int_{K} |x|^2\,d\gamma+\int_{K} |x|^4 \,d\gamma
\\&
\nonumber -\left(n^2 \gamma(K) -2n\int_K |x|^2 \,d\gamma+ \frac{1}{\gamma(K)} \Bigl(\int_K |x|^2\,d\gamma\Bigr)^2\right)
\\&  \nonumber+\frac{1}{n}\left(n^2 \gamma(K) -2n\int_K |x|^2 \,d\gamma+ \frac{1}{\gamma(K)} \Bigl(\int_K |x|^2\,d\gamma\Bigr)^2\right).
\end{align}
Rearranging this inequality, we obtain
$$
\left[\int_{K} |x|^4 \,d\gamma - \frac{1}{\gamma(K)} \Bigl(\int_{K} |x|^2\,d\gamma \Bigr)^2-2\int_{K} |x|^2 \,d\gamma\right]+$$
\begin{equation}\label{2}
\left[-\int_{K} |x|^2 \,d\gamma+\frac{1}{n\gamma(K)} \Bigl(\int_{K} |x|^2\,d\gamma \Bigr)^2\right]\leq 0.
\end{equation}
Recall Lemma 2 from \cite{CFM} (which was a key tool in obtaining the B-theorem):
\begin{equation}\label{CFMeq}
\int_{K} |x|^4 \,d\gamma- \frac{1}{\gamma(K)} \Bigl(\int_{K} |x|^2\,d\gamma\Bigr)^2-2\int_{K} |x|^2 \,d\gamma\leq 0.
\end{equation}
In addition, Lemma \ref{k_2} implies that
\begin{equation}\label{xsq}
-\gamma(K)+\frac{1}{n}\int_K |x|^2\,d\gamma\leq 0.
\end{equation}
Applying (\ref{CFMeq}) and (\ref{xsq}) we obtain the validity of (\ref{2}), which in turn implies the validity of (\ref{maineq-end}).
\end{proof}

As a consequence of Lemma \ref{claim2} and Lemma \ref{cl2}, we confirm the conjecture of Gardner and Zvavitch in the case when $K$ and $L$ are dilates. This result was previously obtained by Gardner and Zvavitch \cite{GZ}, where the authors also used (\ref{CFMeq}). We include the following proposition merely for completeness.

\begin{proposition}\label{dilates}
Let $K$ be a convex set such that $\int_K x d\gamma(x)=0$. Let $L=aK$ for some $a>0$. Then for every $\lambda\in [0,1],$
$$\gamma(\lambda K+(1-\lambda)L)^{\frac{1}{n}}\geq \lambda \gamma(K)^{\frac{1}{n}}+(1-\lambda)\gamma(L)^{\frac{1}{n}}.$$
\end{proposition}
\emph{Proof:} Note that the class $\mathcal{F}$ of dilates of the same convex body is closed under Minkowski convex combinations. Recall, from the proof of Lemma \ref{key_lemma_BM}, that arbitrary $K$ and $L$ can be interpolated by a one-parameter family $\mathcal{K}(h,\psi,I)$ with $h=h_K$ and $\psi=h_L-h_K.$ Recall as well that the boundary condition in the Neumann problem we considered is given by $f(x)=\psi(n_x)=h_L(n_x)-h_K(n_x)$. In the case when $L=aK,$ we are dealing with
$$f(x)=(a-1)h_K(n_x)=(a-1)\langle x, n_x\rangle.$$
By Corollary \ref{equiv-const} and Proposition \ref{raileyprop}, we see that to verify the proposition, is suffices to show that for some $u: K\rightarrow\R$ with
\begin{equation}\label{u}
\langle \nabla u, n_x\rangle =f(x)=(a-1)\langle x, n_x\rangle,
\end{equation}
one has
\begin{equation}\label{verify}
\int_K ||\nabla^2 u||^2+\langle \nabla^2V \nabla u, \nabla u\rangle \,d\gamma \geq \int_K (Lu)^2\,d\gamma-\frac{n-1}{n\gamma(K)}\left(\int_K Lu \,d\gamma\right)^2.
\end{equation}

It remains to note that $u=\frac{a-1}{2} {|x|^2}$ satisfies (\ref{u}), and that Lemma \ref{cl2}, along with the homogeneity of (\ref{verify}), implies the validity of (\ref{verify}) for $u= \frac{a-1}{2}  |x|^2$. $\square$

\begin{remark}
{\rm
Note that Proposition \ref{dilates} implies the validity of the conjecture of Gardner and Zvavitch in dimension 1, since every pair of symmetric intervals are dilates of each other. Furthermore, directly verifying (\ref{maineq-end}) in the case $n=1$ boils down to proving the elementary inequality
$$\alpha(R)=\int_{0}^R (t^4-3t^2) e^{-\frac{t^2}{2}}dt\leq 0,$$
which follows from the fact that $\alpha(0)=\alpha(+\infty)=0$, $\alpha(R)$ decreases on $[0,\sqrt{3}]$ and increases on $[\sqrt{3},+\infty]$. It of course also follows from (\ref{CFMeq}) and (\ref{xsq}), but that would be an overkill.

It is curious to note that Lemma \ref{lemma1} is also sharp when $n=1$: for every $u:[-R,R]\rightarrow \R$ with $Lu=1$ and with the boundary condition $u'(R)=-u'(-R)$, one has
$$\beta(R)=\frac{\int_{-R}^{R} [(u'')^2+(u')^2]e^{-\frac{t^2}{2}}dt}{\int_{-R}^{R} e^{-\frac{t^2}{2}}dt}\geq 1.$$
In fact, equality is never attained unless $R=0,$ and $\lim_{{R\rightarrow 0}} \beta(R)= 1.$ A routine computation shows that $\beta(R)$ is strictly increasing in $R$, and $\lim_{{R\rightarrow\infty} } \beta(R)= \infty$. Furthermore, $\beta(R)$ grows very fast.

This indicates that our proof of Lemma \ref{lemma2} is sub-optimal, at least in the case $n=1$: we replace the term which includes $|\nabla u|^2$ with the much smaller term, while $|\nabla u|^2$ has large growth. The constant ${1}/{2}$ which we get after such replacement is attained when $R=\infty$, and in fact the estimate decreases as $R$ increases, contrary to the actual behavior of $\beta(R).$
}
\end{remark}

\begin{remark}\label{TN}{\rm Consider $p=0.5C_0(\gamma_2,\mathcal{F}_o)$, that is the largest number such that for any pair of convex sets $K$ and $L$ in $\R^2$ containing the origin, and for any $\lambda\in [0,1]$,
$$\gamma(\lambda K+(1-\lambda)L)^p\geq \lambda \gamma(K)^p+(1-\lambda)\gamma(L)^p.$$
Nayar and Tkocz \cite{NaTk} showed that $p<0.5$. Furthermore, using their argument, one may observe that $p\leq 1-\frac{2}{\pi}\approx 0.363.$ Our results imply that $p\geq 0.25.$ Therefore, 
$$p\in [0.25,0.363].$$

We would like to add that our proof, in fact, indicates a bit more:  
\begin{equation}\label{est}
p\geq \inf_{K\in \mathcal{F}_o}\frac{1}{\gamma(K)}\int_K \frac{1}{x^2+2}d\gamma(x).
\end{equation}
A numerical computation shows that
\begin{equation}\label{est-1}
\frac{1}{2\pi}\int_{\R^2} \frac{e^{-\frac{x^2}{2}}dx}{x^2+2}\approx 0.298.
\end{equation}
Using Caffarelli's contraction theorem \cite{Caf}, one may observe that for any symmetric (and not just containing the origin) convex set $K$ in $\R^2$,
\begin{equation}\label{est-2}
\frac{1}{\gamma(K)}\int_K \frac{1}{x^2+2}d\gamma(x)\geq 0.298.
\end{equation}
Indeed, Caffarelli's theorem guarantees the existence of a 1-Lipschitz map $T$ which pushes forward the Gaussian measure to its restriction on $K$. In case $K$ is symmetric, one has $T(0)=0,$ and the 1-Lipschitz property yields $|T(x)|=|T(x)-T(0)|\leq |x-0|=|x|$. Hence the normalized integral from (\ref{est-2}) is greater than the corresponding integral over $\R^2$. 

However, this does not provide an insight into calculating the infimum from (\ref{est}), which runs over the class of all convex sets containing the origin. In any case, this infimum does not exceed $0.298$, and therefore it is certainly smaller than $0.363$.

In conclusion, unfortunately, combining our estimate with the example of Nayar and Tkocz \cite{NaTk}, one may not determine the value of $p$ explicitly.}
\end{remark}

\end{document}